\documentclass[aip,jmp,preprint]{revtex4-2}

\usepackage{amsmath,amssymb,amsthm}

\usepackage{mathabx}
\usepackage{mathabx}
\usepackage{graphicx}
\usepackage{hyperref}
\usepackage{mathtools}
\usepackage{appendix}
\usepackage{lineno}

\renewcommand{\leq}{\leqslant}

\renewcommand{\L}{\mathcal{T}}

\newcommand{\R}{\mathbb{R}}

\newcommand{\C}{\mathbb{C}}

\newtheorem{theorem}{Theorem}[section]
\newtheorem{lemma}[theorem]{Lemma}

\newtheorem{definition}[theorem]{Definition}

\newtheorem*{main-theorem}{Main Theorem}
\newtheorem*{remark*}{Remark}
\newtheorem*{case*}{Case}

\numberwithin{equation}{section}

\begin{document}

\title{Spectral stability of periodic traveling waves in Caudrey-Dodd-Gibbon-Sawada-Kotera Equation}

\author{Sudhir Singh\textsuperscript{*}}
\email{ssingh8@gitam.edu}
\affiliation{Department of Mathematics, GITAM University, Bangalore 560012, India}

\author{Nitesh Sharma}
\email{niteshs@iiitd.ac.in}
\affiliation{Department of Mathematics, IIIT Delhi, New Delhi 110020, India}

\author{Ashish Kumar Pandey}
\email{ashish.pandey@iiitd.ac.in}
\affiliation{Department of Mathematics, IIIT Delhi, New Delhi 110020, India}

\date{\today}

\begin{abstract}
We study the spectral stability of the one-dimensional small-amplitude periodic traveling wave solutions of the (1+1)-dimensional Caudrey-Dodd-Gibbon-Sawada-Kotera equation. We show that these waves are spectrally stable with respect to co-periodic as well as square integrable perturbations. 
\end{abstract}

\maketitle

\section{Introduction}
The Caudrey--Dodd--Gibbon--Sawada--Kotera (CDG--SK) equation is a nonlinear fifth-order evolution equation  written as
\begin{equation}\label{eq:CDGSK}
    u_t+u_{xxxxx}+15 (uu_{xx}+u^3)_x=0,
\end{equation}
which is a member of the KdV-type integrable hierarchy\cite{caudrey1976new}.
It was discovered in the 1970s as one of the analogs of the Korteweg--de Vries (KdV) equation that exhibits complete integrability\cite{sawada1974method}.
Sawada and Kotera introduced a ``KdV-like'' fifth-order equation and constructed $N$-soliton solutions\cite{sawada1974method}, while Caudrey, Dodd, and Gibbon derived the CDG equation within a new integrable KdV hierarchy\cite{caudrey1976new}.
These equations are related via Miura-type transformations and share integrability properties like Lax pairs, conservation laws, and soliton solutions.
Integrability was further confirmed by B\"acklund transformations\cite{satsuma1977backlund} and Painlev\'e analysis\cite{weiss1984classes}.

\subsection{Integrability}
The CDG--SK equation admits  Lax pair formulations.
Kaup introduced a third-order spectral problem to derive the SK/CDG equation as part of the inverse scattering framework\cite{Kaup1980}.
Caudrey et al.\cite{caudrey1976new} and Dodd and Gibbon\cite{dodd1978prolongation} provided a $2 \times 2$ matrix Lax pair and analyzed the equation's prolongation structure.
 Fuchssteiner \& Oevel proposed the bi-Hamiltonian formalism and conserved covariants\cite{fuchssteiner1982bi}, and Aiyer et al.\cite{aiyer1986solitons} constructed the recursion operator.
Konno studied the conservation laws\cite{konno1992conservation}, and Lou studied a larger family of symmetries\cite{lou1993twelve}. 

\subsection{Solitons}
Sawada and Kotera constructed explicit soliton solutions via Hirota's method\cite{sawada1974method}, with further contributions from Satsuma and Kaup using B\"acklund transformations\cite{satsuma1977backlund}.
Wazwaz\cite{wazwaz2008n} and Alam et al. \cite{alam2021determination} applied symbolic computation methods.
Soliton interactions remain elastic and consistent with integrability.
Ma et al.\cite{ma2023phase} reported transitions to breather and soliton molecule structures through the Hirota bilinear method. 

\subsection{Periodic wave solutions}
Periodic (cnoidal) wave solutions were derived using elliptic functions and finite-gap methods.
Tian \&  Zhang\cite{tian2010riemann} gave Riemann theta function solutions, while  Ram\'{\i}rez et al. obtained closed-form periodic solutions through symbolic computations \cite{ramirez2016reductions}. Jun \&  Ji obtained the elliptic solutions via linear superposition approach \cite{li2004different}.
These results confirm the existence of closed-form periodic solutions, which can be interpreted as a continuous family bridging solitons and linear waves.

\subsection{Well-posedness}
Kaup outlined the inverse scattering transform (IST) solution of the CDG--SK equation\cite{Kaup1980}.
Smooth decaying initial data leads to global well-posedness due to integrability.
 Dye and Parker\cite{dye2001bidirectional} studied bidirectional variants of the equation.
While detailed modern Sobolev-space well-posedness results are limited, energy methods and integrability suggest local and global existence for smooth data.

\subsection{Stability results for fifth-order models.}
The stability and instability of periodic, solitary, and compacton solutions are investigated for dispersive fifth-order water wave models within the KdV family. Particular emphasis is placed on the rigorous analysis of solitary wave stability across various members of this fifth-order KdV family. Notably, a numerical framework is employed to examine the stability and instability of solitary waves in the fifth-order KdV equation \cite{bridges2002stability}. Orbital stability methods for Hamiltonian systems are used to analyze a fifth-order evolutionary model proposed by Il'ichev, and Semenov \cite{il1992stability}, while Tan et al.\ \cite{tan2002semi} explore the stability of embedded solitons in the Hamiltonian fifth-order KdV equation. Esfahani and Levandosky \cite{esfahani2021existence} investigate the stability of a fifth-order model using a variational approach. Additionally, Natali \cite{natali2010note} studies the nonlinear stability of both the Kawahara–KdV and modified Kawahara–KdV equations. The stability of compactons is studied by  Dey \& Khare for dispersive $ K(m,n,p)$ type equations \cite{dey2000stability}.


One of the classical higher-dispersion analogs of the KdV equation, the fifth-order dispersive model known as the Kawahara equation, serves as a prototype for rigorous studies on linear (spectral), nonlinear, modulational, high-frequency, orbital, numerical, and analytical stability and instability of periodic, solitary, and transient waves of small or arbitrary amplitude. For instance, the Evans function approach was employed by Bridges \& Derks to analyze the linear stability of solitary waves in the Kawahara model \cite{bridges2002linear}; Haragus et al.\ studied the spectral stability of spatially periodic waves \cite{haragus2006spectral}; and the orbital stability of solitary waves was investigated by Kabakouala \& Molinet \cite{kabakouala2018stability}. Andrade et al.\ examined the orbital stability of periodic traveling waves \cite{de2017orbital}, while high-frequency instabilities were explored via perturbative methods in \cite{trichtchenko2018stability, creedon2021high}. The nonlinear stability and instability of Kawahara-type fifth-order models were rigorously studied by Quintero \& Mu\~{n}oz \cite{quintero2016analytic}. Recently, Creedon demonstrated that the Kawahara model admits exotic Wilton ripples \cite{creedon2024existence}.

Deconinck and Kapitula\cite{kapitula2015spectral, DK} proved the orbital stability of KdV's periodic waves, motivating similar expectations for CDG--SK.
Though no explicit stability analysis exists for periodic CDG--SK waves, spectral stability of solitons has been rigorously demonstrated by Wang\cite{wang2023spectral}.
The integrable structure and conservation laws suggest spectral and orbital stability for periodic traveling waves, although further investigation is needed. Motivated by the existence of exotic waves and stability and instability phenomena in fifth-order models and their hidden integrability structures, this article establishes the existence of periodic traveling wave solutions of \eqref{eq:CDGSK}, derives their small-amplitude asymptotics and investigates their spectral stability.

The structure of the article is as follows. In the next section, we discuss the existence of periodic traveling waves, followed by the proof of the main theorem. Additional details and proofs are provided in the Appendix. We state main theorems of this article below.

\begin{theorem}[Existence of small periodic traveling waves]\label{thm:existence}
Fix a wavenumber $k>0$. There exists $\varepsilon_0>0$ such that for each $0<|a|<\varepsilon_0$, there is a $2\pi/k$-periodic traveling wave solution of \eqref{eq:CDGSK} of small amplitude $a$. Specifically, there exists a $2\pi$-periodic function $w(z)$ and a speed $c$ of the form 
\[
    w(z) = a\,w_1(z) + a^2\,w_2(z) + a^3\,w_3(z) + O(a^4), \qquad
    c = c_0 + c_2\,a^2 + O(a^4),
\] 
with $c_0 = k^4$, and constants $c_2\in \mathbb{R}$, such that $u(x,t) = w(k(x-ct))$ is a solution of \eqref{eq:CDGSK}. In particular, one can choose the expansion so that $w(z)$ is an even $2\pi$-periodic function with zero mean. The leading profile and speed coefficients are 
\[
    w_1(z) = \cos(z), \quad  w_2(z) = -\frac{15}{2k^2} + \frac{1}{2k^2}\cos(2z),\,\quad w_3(z)=\frac{3}{16k^4}\cos(3z),
\] 
\[
\quad c_0 = k^4, \quad c_2 = 105\,.
\]
Consequently, as $a\to0$, $w(z)$ converges (in $C^\infty$) to $0$ and $c$ converges to $k^4$. Furthermore, the map $a \mapsto (w(\cdot;a), c(a))$ can be chosen to be smooth (in fact, real-analytic) for $|a|$ sufficiently small.
\end{theorem}

\begin{theorem}[Spectral stability]\label{thm:spec} 
    For sufficiently small $|a|$, periodic traveling waves $u(x,t) = w(k(x-ct))$ of \eqref{eq:CDGSK} given in Theorem~\ref{thm:existence} are spectrally stable with respect to localized ($L^2(\mathbb{R})$) or co-periodic perturbations ($L^2(\mathbb{T})$).  
\end{theorem}

\subsection*{Notations}\label{sec:notations}
Through out the article, we have used the following notations. 
Here, $L^{2}(\mathbb{R})$ is the set of Lebesgue measurable, real or complex-valued functions over $\mathbb{R}$ such that
$$
\|f\|_{L^{2}(\mathbb{R})}=\left(\int_{\mathbb{R}}|f(x)|^{2} d x\right)^{1 / 2}<+\infty,
$$
and, $L^{2}(\mathbb{T})$ denote the space of $2 \pi$-periodic, measurable, real or complex-valued functions over $\mathbb{R}$ such that
$$
\|f\|_{L^{2}(\mathbb{T})}=\left(\frac{1}{2 \pi} \int_{0}^{2 \pi}|f(x)|^{2} d x\right)^{1 / 2}<+\infty.
 $$
For $s \in \mathbb{R}$, let $H^{s}(\mathbb{R})$ consists of tempered distributions such that
$$
\|f\|_{H^{s}(\mathbb{R})}=\left(\int_{\mathbb{R}}\left(1+|t|^{2}\right)^{s}|\hat{f}(t)|^{2} d t\right)^{\frac{1}{2}}<+\infty,
$$
and
$$
H^{s}(\mathbb{T})=\left\{f \in H^{s}(\mathbb{R}): f \text { is } 2 \pi \text {-periodic }\right\}.
$$
We define $L^{2}(\mathbb{T})$-inner product as
\begin{equation}\label{eq:inn}
\langle f, g\rangle=\frac{1}{ \pi} \int_{0}^{2 \pi} f(z) \bar{g}(z) d z=\sum_{n \in \mathbf{Z}} \hat{f}_{n} \overline{\hat{g}_n},
\end{equation}
where $\widehat{f}_{n}$ are Fourier coefficients of the function $f$ defined by
$$
\widehat{f}_{n}=\frac{1}{2 \pi} \int_{0}^{2 \pi} f(z) e^{i n z} d z.
$$
Throughout the article, $\Re(\mu)$ represents the real part of $\mu\in\mathbb{C}$.

\section{Proof of 
Theorem~\ref{thm:existence}}\label{sec:existence}
We look for a periodic traveling wave solution of \eqref{eq:CDGSK} and therefore, set $u(x,t)=w(z)$, where $w$ is a $2\pi$-periodic function of its argument, and $z=k(x-c t)$ with $c>0$ being the speed of the traveling wave. We obtain a fifth order ODE in $w$ given by
\begin{equation}\label{eq:ODE5}
    -c\,k\,w' +k^5\,w^{(5)} + 15\,k^3\,w\,w^{(3)} + 15\,k^3\,w'\,w'' + 45\,k\,w^2 w' = 0,
\end{equation}
where $w^{(n)}$ denotes the $n$-th derivative with respect to $z$. 
We integrate it once with respect to $z$ and assume the integration constant is zero (this corresponds to selecting the mean of $w$ appropriately for periodic solutions). This yields a fourth-order ODE
\begin{equation}\label{eq:ODE4}
    k^4\,w^{(4)} - c\,w + 15\,k^2\,w\,w'' + 15w^3 = 0,
\end{equation}
which $w(z)$ must satisfy for a $2\pi$-periodic traveling wave.

\subsection*{Linearization and kernel decomposition}
We first analyze the linearization
\begin{equation}\label{eq:linear}
    L_c h:=k^4(\partial_z^4 - c)h = 0.
\end{equation}
of \eqref{eq:ODE4} at the trivial solution $w=0$. We seek $2\pi$-periodic solutions to \eqref{eq:linear}. Using the trial solution $h(z)=e^{i m  z}$ (with integer $m$), we obtain the dispersion relation $-(c) + (m k)^4 = 0$, i.e.\ $c = (m k)^4$. The smallest positive $m$ giving a nontrivial solution is $m=1$. Thus the \emph{critical} (smallest) phase speed for which \eqref{eq:linear} admits a nontrivial periodic solution is 
\[ c_0 = (1\cdot k)^4 = k^4. \] 
At $c=c_0$, the linear operator $L_{c_0}$ has a nontrivial kernel spanned by $\cos(z)$ and $\sin(z)$. (Equivalently, $e^{\pm i z}$ are neutral eigenmodes.) No other harmonics resonate at this parameter value since for $|m|\neq 1$, $(m k)^4 - k^4 \neq 0$. We also note that for $c=c_0$, the constant (zero-frequency) mode is not in the kernel because $L_{c_0}1 = -k^4 \neq 0$.

The two-dimensional kernel at $c=c_0$ reflects two symmetries of the problem: a phase translation symmetry (shifting $z$) and a spatial reflection symmetry ($z \mapsto -z$). The translation symmetry implies that if $w(z)$ is a solution, then so is $w(z+\Delta)$ for any constant $\Delta$, which leads to the $\sin(z)$ mode (since $\sin(z)$ is the derivative of $\cos(z)$ up to scaling). The reflection (even/odd) symmetry implies we can choose solutions that are either even or odd. We will select the solution to be even in $z$ for definiteness. This breaks the translational degeneracy by fixing the phase so that $w'(z)$ has zero mean and is $\pi$-antiperiodic, eliminating the $\sin(z)$ component. In practice, we impose that $w(z)$ is an even function, which means we take $w_1(z) = \cos(z)$ as the normalized fundamental mode. 

Now we formulate the problem as a bifurcation equation. We treat $c$ as an unknown parameter that must be adjusted to find nontrivial $w$. For a fixed $k>0$, we write \eqref{eq:ODE4} as $F(w,c)=0$ for the function
\[ 
    F(w,c) := k^4\,w^{(4)} - c\,w + 15\,k^2\,w\,w'' + 15w^3\,,
\] 
defined on the space $H^4_{\text{even}}(\mathbb{T})\times \mathbb{R}^+$ where $H^4_{\text{even}}(\mathbb{T})$ is the subspace of $H^4(\mathbb{T})$ consisting of even functions. We seek a solution $(w,c)$ near $(0,c_0)$. The linearized operator at $(0,c_0)$ is $D_w F(0,c_0)h =L_{c_0}h= k^4\,h^{(4)} - c_0\,h$. As discussed, $\ker(L_{c_0}) = \text{span}\{\cos(z)\}$ in $H^4_{\text{even}}(\mathbb{T})$. We decompose the function space as 
\[ 
H^4_{\text{even}}(\mathbb{T}) = \ker(L_{c_0}) \oplus \mathrm{Range}(L_{c_0}),
\] 
i.e.\ any $2\pi$-periodic even function $w(z)$ can be uniquely written as 
\[ 
    w(z) = a\,\cos(z) + h(z),
\] 
with $a\in\mathbb{R}$ and $h \in \mathrm{Range}(L_0)$ such that $\langle h,\cos(z)\rangle_{L^2}=0$. Here $a$ will serve as our small amplitude parameter. We also define the projection $P$ onto $\ker L_{c_0}$ along $\mathrm{Range}(L_{c_0})$ by 
\[ 
    P[w] = \Big(\frac{1}{\pi}\int_{0}^{2\pi}w(z)\cos(z)\,dz\Big)\,w(z),
\] 
and $Q = I - P$ as the complementary projection onto $\mathrm{Range}(L_{c_0})$. In particular, $P[\cos(z)] = \cos(z)$ and $Q[\cos(z)]=0$.

Applying these projections to the equation $F(w,c)=0$ yields two equations:
\begin{align}
    P[F(a\cos(z) + h,\,c)] &= 0, \label{eq:LS1}\\
    Q[F(a\cos(z) + h,\,c)] &= 0. \label{eq:LS2}
\end{align}
Equation \eqref{eq:LS1} is the bifurcation equation (the solvability condition) in the one-dimensional kernel, and \eqref{eq:LS2} is an equation in the range which can be solved for $h$ given $a$ and $c$. We will solve \eqref{eq:LS2} by the implicit function theorem for $h = h(a,c)$, and then solve \eqref{eq:LS1} for $a$ and $c$.

\subsection*{Solving the range equation}
For $(w,c)$ near $(0,c_0)$, the linear operator $L_{c_0}$ is invertible on $\mathrm{Range}(L_{c_0})$ (since $L_{c_0} h = 0$ implies $h\in\ker L_{c_0}$ which is orthogonal to the range). Moreover, by standard Fourier series theory, this inverse is bounded on the space of $2\pi$-periodic functions (excluding the kernel). Thus we can solve the $Q$-equation \eqref{eq:LS2} for $h$ as a smooth function of $(a,c)$ near $(0,c_0)$. In fact, by the implicit function theorem in Banach spaces, there exists a smooth mapping $h = H(a,c)$ with $H(0,c_0)=0$ and $\nabla_{a,c}H(0,c_0)=0$, such that for all sufficiently small $a$ and $c$,
\[ 
    Q[F(a\cos(z) + H(a,c),\,c)] = 0.
\] 
Substituting $h = H(a,c)$ back into \eqref{eq:LS1}, we obtain a scalar equation in the unknowns $a$ and $c$:
\begin{equation}\label{eq:red-eq}
    P[F\!\big(a\cos(z) + H(a,c),\,c\big)] = 0.
\end{equation}
This is the bifurcation equation (or reduced equation) governing the small solutions. It can be expanded in powers of $a$ (which is small) to determine nontrivial solutions perturbatively. Any solution $(a,c)$ of \eqref{eq:red-eq} with $a\neq0$ yields a solution $w = a\cos(z) + H(a,c)$ of the full problem.

\subsection*{Asymptotic expansion}
To solve \eqref{eq:red-eq}, we expand $H(a,c)$ and $c$ as power series in the small parameter $a$. We write 
\[ 
    c = c_0 + c_1 a + c_2 a^2 + c_3 a^3 + O(a^4),
\] 
and similarly expand the correction $H(a,c) = h_1 a + h_2 a^2 + h_3 a^3 + O(a^4)$, where each $h_j(z)$ is a $2\pi$-periodic even function orthogonal to $\cos(z)$. We substitute these expansions into \eqref{eq:red-eq} and collect terms by powers of $a$. Each order will yield an equation that can be solved for the unknown coefficients $c_j$ and $h_j$.

At $O(a^1)$, The $a$-linear terms in \eqref{eq:red-eq} give 
\[ 
    P\big[-k^4 \cos(z)-k^4 h_1+k^4 (\cos(z)+h_1^{(4)})\big] = 0,
\] 
which leads to $P\big[h_1-h_1^{(4)}\big]=0$. We take $h_1=0$ as one of the solutions of this equations. 

At $O(a^2)$, The $a^2$-terms in \eqref{eq:red-eq} yield
\[
    P\Big[-k^4(h_2-h_2^{(4)})-c_1 \cos(z)-15k^2 \cos^2(z)\Big] = 0.
\]
We set $-k^4(h_2-h_2^{(4)})-c_1 \cos(z)-15k^2 \cos^2(z)=0$ to find that $c_1=0$ and $k^2(h_2^{(4)}-h_2)=15 \cos^2(z)$ which yields 
\[
    h_2(z) = -\frac{15}{2k^2} + \frac{1}{2k^2}\cos(2z)\,. 
\] 
At $O(a^3)$, The $a^3$ terms in \eqref{eq:red-eq} produce 
\[
    P\Big[\,k^4(h_3^{(4)}-h_3)-c_2\cos(z)+15k^2(h_2^{(2)}-h_2)\cos(z)+15\cos^3(z)\,\Big] = 0.
\] 
We set coefficient of $\cos(z)$ equal to zero to obtain
\[
c_2=105
\]
and $\cos(3z)$ equal to zero to obtain
\[
h_3(z)=\frac{3}{16k^4}\cos(3z).
\]
We have thus determined all unknowns up to third order.

\subsection*{Higher-order terms and smoothness}
The process described above can be continued to higher orders $O(a^4)$, $O(a^5)$, etc., in principle yielding all coefficients $\{c_j, w_j(z)\}$ recursively. In practice, the algebra grows increasingly complicated. However, one can make several general observations. First, all odd-order terms $c_{2n+1}$ will turn out to vanish ($c_1=0$, and one can show similarly $c_3=0$ by examining the $O(a^4)$ equations, due to symmetry and the structure of resonances). Thus $c(a)$ is an even function of $a$, consistent with the reversibility (reflection symmetry) of the ODE \eqref{eq:ODE4}. Second, at each order $n\ge 2$, the solvability condition (projection onto $\cos(z)$) will determine $c_n$ in terms of lower-order quantities, ensuring a unique power series solution for $c$. This means the small-amplitude branch of solutions is uniquely determined (up to the overall phase choice) by the amplitude parameter $a$. Finally, from \eqref{eq:ODE4}, we see that
\[
k^4 w^{(4)}=c w-15(k^2 w w''+w^3)
\]
and therefore, since $w\in H^4(\mathbb{T})$, $w^{(4)}\in H^2(\mathbb{T})$ and in turn, $w\in H^6(\mathbb{T})$. By a bootstrapping argument, we obtain that $w$ is smooth. (For a detailed general proof of existence using Lyapunov--Schmidt reduction and analytic implicit function theorem, see  \cite{IoossKirch1990}.)

\section{Proof of Theorem~\ref{thm:spec}}\label{sec:spec}
In this section, we study the spectral stability of periodic traveling waves of \eqref{eq:CDGSK} obtained in Thoreem~\ref{thm:existence} and provide a proof of Theorem~\ref{thm:spec}. We seek a perturbed solution $u(x,t)=w(z)+v(z,t)$, $z=k(x-ct)$ of \eqref{eq:CDGSK} where $v(z,t)$ is a small perturbation and arrive at the linearized equation
\begin{equation}
    v_t-ckv_{z}+k^5 v_{zzzzz}+15k(k^2(wv_{zz}+vw_{zz})+3w^2v)_{z}=0.
\end{equation}
We seek a solution of the form $v(z,t) = e^{{k\lambda}t} \tilde{v}(z), \ \lambda \in \mathbb{C},$ to obtain
\begin{equation}\label{eq:spec}
\mathcal{T}_{a} \tilde{v} :=  \lambda \tilde{v},
\end{equation}
where
\begin{equation}\label{eq:Ta}
\mathcal{T}_{a}  :=  \partial_z \left( c - k^4 \partial_z^4 - 15 \left(k^2 \left( w \partial_z^2 + w_{zz} \right) + 3 w^2 \right)\right).
\end{equation}

The operator $\mathcal{T}_{a}$ is defined on $L^2(\mathbb{R})$ with dense domain $H^5(\mathbb{R})$. We define the spectral stability of the periodic traveling wave solution $w$ as follows:

\begin{definition}\label{def:stab}
    The solution $u(x,t)=w(z)$, $z=k(x-ct)$ of \eqref{eq:CDGSK} obtained in Theorem~\ref{thm:existence}  is spectrally stable if $L^2(\mathbb{R})$-spectrum of $\mathcal{T}_{a}$ consists of $\lambda \in \mathbb{C}$ with $\Re(\lambda) \leq 0$, otherwise, it is deemed to be spectrally unstable.
\end{definition}

Since \(\mathcal{T}_{a}\) is a real operator, if \(\lambda\) is an eigenvalue with eigenfunction \(\tilde{v}(z)\), then \(\overline{\lambda}\) is also an eigenvalue with eigenfunction \(\overline{\tilde{v}(z)}\).  The operator \(\mathcal{T}_{a}\) is reversible under the transformation \(z \to -z\). Specifically, if \(\tilde{v}(z)\) is an eigenfunction with eigenvalue \(\lambda\), then \(\tilde{v}(-z)\) is also an eigenfunction with eigenvalue \(-\lambda\). Combining, if $\lambda$ is in the spectrum of \(\mathcal{T}_{a}\) then so are $-\lambda, \overline{\lambda}$, and $-\overline{\lambda}$.  Therefore, the spectrum is symmetric with respect to both real and imaginary axes. Following Definition~\ref{def:stab}, the solution is unstable if there is any $\lambda$ in spectrum off imaginary axis.

The operator $\mathcal{T}_{a}$ has periodic coefficients and therefore, can't have decaying eigenfunctions in $L^2(\mathbb{R})$. Hence, $\mathcal{T}_{a}$ has continuous spectrum in $L^2(\mathbb{R})$. Using Bloch Transform, for every $\tilde{v}\in L^2(\mathbb{R})$ there exists a unique family, $v_\xi\in L^2(\mathbb{T})$, $\xi\in (-1/2,1/2]$ (called Floquet exponent), such that
\[
\tilde{v}(z)=\int_{-1/2}^{1/2} v_\xi(z) e^{i\xi z}~d\xi.
\]
Using this, we obtain that
\begin{equation*}
\text{spec}_{L^2(\mathbb{R})}(\mathcal{T}_{a})=
\bigcup_{\xi\in(-1/2,1/2]}\text{spec}_{L^2(\mathbb{T})}(\mathcal{T}_{a,\xi}),
\end{equation*}
where
\begin{equation}\label{E:Lxi}
\mathcal{T}_{a,\xi}:=e^{-i\xi z}\mathcal{T}_{a} e^{i\xi z}
\end{equation}
is defined on $L^2(\mathbb{T})$ and $L^2(\mathbb{T})$-spectrum of $\L_\xi$ comprises of discrete eigenvalues of finite multiplicities.

A straightforward calculation shows that for trivial solution ($w\equiv 0$), we have
\begin{equation}\label{eq:T0xi}
    \mathcal{T}_{0,\xi}  = k^4 (\partial_z+i\xi)(1-(\partial_z+i\xi)^4)
\end{equation}
and its eigenvalues are given by
\begin{equation}\label{eq:T0ei}
    \mathcal{T}_{0,\xi} e^{inz} = i \omega_{n,\xi} e^{inz}, \quad n \in \mathbb{Z},
\end{equation}
where
\begin{equation}
    \omega_{n,\xi} = k^4 (n + \xi) (1-({n + \xi})^4).
\end{equation}
We have $\sigma (\mathcal{A}_{0,\xi}) \subset i \mathbb{R}$ which should be the case since $a = 0$ corresponds to the zero solution. 

Notice that 
\[
||\mathcal{T}_{a,\xi}-\mathcal{T}_{0,\xi}||_{L^2(\mathbb{T})\to L^2(\mathbb{T})}\to 0 \; \text{as} \;|a|\to 0
\]
and therefore, for $|a|\ll 1$, eigenvalues of $\mathcal{T}_{a,\xi}$ continuously bifurcate from that $\mathcal{T}_{0,\xi}$. Observe that because of quad-fold symmetry of the spectrum, eigenvalues $i\omega_{n,\xi}$ of $\mathcal{T}_{0,\xi}$ leave imaginary axis for small $|a|$ only when they collide on imaginary axis for some $\xi$.

To observe the collision among eigenvalues $i\omega_{n,\xi}$, let \(y = n + \xi\). Since \(n\) is an integer and \(\xi \in \left(-\frac{1}{2}, \frac{1}{2}\right]\), the value of \(y\) lies in the interval
\[
y \in \left(n - \frac{1}{2}, n + \frac{1}{2}\right].
\]
The equation becomes:
\[
\omega_{n,\xi} = k^4 y \left(1 - y^4\right).
\]
If \(\xi = 0\), then \(x = n\), and the equation simplifies to
\[
\omega_{n,0} = k^4 n \left(1 - n^4\right).
\]
For \(\omega_{n,0} = \omega_{m,0}\), we require,
\[
n \left(1 - n^4\right) = m \left(1 - m^4\right).
\]
The only solutions are \(n = m\) or \(n = 1, -1, 0\). For example,
\begin{itemize}
    \item If \(n = 1\), then \(\omega_{1,0} = k^4 \cdot 1 \cdot (1 - 1^4) = 0\).
    \item If \(n = -1\), then \(\omega_{-1,0} = k^4 \cdot (-1) \cdot (1 - (-1)^4) = 0\).
    \item If \(n = 0\), then \(\omega_{0,0} = k^4 \cdot 0 \cdot (1 - 0^4) = 0\).
\end{itemize}
Thus, 
\begin{equation}\label{eq:eig0}
\omega_{1,0} = \omega_{-1,0} = \omega_{0,0} = 0.
\end{equation}
If \(\xi \neq 0\), then \(y = n + \xi\) is not an integer. For \(\omega_{n,\xi} = \omega_{m,\xi}\), we require
\[
(n + \xi) \left(1 - (n + \xi)^4\right) = (m + \xi) \left(1 - (m + \xi)^4\right).
\]
This is a nonlinear equation in \(\xi\), and for \(\xi \in \left(-\frac{1}{2}, \frac{1}{2}\right]\), it has no solutions unless \(n = m\). This is because the function \(f(y) = y(1 - y^4)\) is strictly monotonic in the intervals \(\left(n - \frac{1}{2}, n + \frac{1}{2}\right)\) for \(n \neq 0, 1, -1\), ensuring distinct values of \(\omega_{n,\xi}\) for distinct \(n\). Therefore, there is only one collision of eigenvalues on the origin. give in \eqref{eq:eig0}, which we will analyze for stability in what follows.

Clearly $\lambda = 0$ is an isolated eigenvalue of $\mathcal{T}_{0,0}$ with algebraic and geometric multiplicity three and
\begin{equation}
    \ker \left(\mathcal{T}_{0,0}\right) = \text{span} \left\{ \cos(z), \sin(z), \frac{1}{\sqrt{2}} \right\}.
\end{equation}

\begin{lemma}\label{l:31}
    For any $a$ and $\xi$ sufficiently small, the following properties hold
    \begin{enumerate}
        \item The spectrum of $\mathcal{T}_{a,\xi}$ decomposes as 
        \begin{equation*}
           \operatorname{spec}(\mathcal T_{a,\xi})= \operatorname{spec}_0(\mathcal T_{a,\xi})\cup\operatorname{spec}_1(\mathcal T_{a,\xi})
        \end{equation*}with
        \begin{equation*}
           \operatorname{spec}_0(\mathcal T_{a,\xi})\subset B(0;R/3),\quad \operatorname{spec}_1(\mathcal T_{a,\xi})\subset \C\setminus \overline{B(0;R/2)},
        \end{equation*}where $R=5k^4$.
        \item The spectral projection $\mathcal{P}_{a,\xi}$ associated with $\operatorname{spec}_0(\mathcal T_{a,\xi})$ 
        \begin{equation}\label{eq:P}
        \mathcal{P}_{a,\xi}=\dfrac{1}{2\pi i}\int_{\partial B(0;R/3)}(\lambda-\mathcal{T}_{a,\xi})^{-1}d\lambda
        \end{equation}
        satisfies $\|\mathcal{P}_{a,\xi}-\mathcal{P}_{0,0}\|=O(|\xi|+|a|)$. The operators $\mathcal{P}_{a,\xi}$ are well defined projectors commuting with $\mathcal{T}_{a,\xi}$, that is, 
        \begin{equation}
\mathcal{P}^2_{a,\xi}=\mathcal{P}_{a,\xi},\quad\mathcal{P}_{a,\xi}\mathcal{T}_{a,\xi}=\mathcal{T}_{a,\xi}\mathcal{P}_{a,\xi}.
        \end{equation}       
        \item $\operatorname{spec}_1(\mathcal T_{a,\xi})\subset i \R$
        \item The projectors $\mathcal{P}_{a,\xi}$ are similar one to each other: the transformation operators
        \begin{equation}\label{e:u}
        \mathcal{U}_{a,\xi}:=(\mathcal{I}-(\mathcal{P}_{a,\xi}-\mathcal{P}_{0,0})^2)^{-1/2}[\mathcal{P}_{a,\xi}\mathcal{P}_{0,0}+(\mathcal{I}-\mathcal{P}_{a,\xi})(\mathcal{I}-\mathcal{P}_{0,0})]
        \end{equation} 
        are bounded and invertible in $H^s(\mathbb{T})$ and in $L^2(\mathbb{T})$, with inverse
        \begin{equation}\label{e:u1}
\mathcal{U}^{-1}_{a,\xi}=[\mathcal{P}_{0,0}\mathcal{P}_{a,\xi}+(\mathcal{I}-\mathcal{P}_{0,0})(\mathcal{I}-\mathcal{P}_{a,\xi})](\mathcal{I}-(\mathcal{P}_{a,\xi}-\mathcal{P}_{0,0})^2)^{-1/2},
        \end{equation}
        and 
        \begin{equation}\label{e:u2}
\mathcal{U}_{a,\xi}\mathcal{P}_{0,0}\mathcal{U}^{-1}_{a,\xi}=\mathcal{P}_{a,\xi},\quad \mathcal{U}^{-1}_{a,\xi}\mathcal{P}_{a,\xi}\mathcal{U}_{a,\xi}=\mathcal{P}_{0,0}.
\end{equation}
        \item The subspaces $V_{a,\xi}=\text{Range}(\mathcal{P}_{a,\xi})$ are isomorphic one to each other, $V_{a,\xi}=\mathcal{U}_{a,\xi}V_{0,0}$. In particular, for any $a$ and $\xi$ sufficiently small, $\dim{V_{a,\xi}}=\dim{V_{0,0}}=3$.
    \end{enumerate}
 \end{lemma}
 \begin{proof}
     See the proof of Lemma~3.1 in \cite{masperogkdv} in a similar situation.
 \end{proof}
Lemma~\ref{l:31} shows that the operator $\mathcal{U}_{a,\xi}$ is an isomorphism between $V_{0,0}$ and $V_{a,\xi}$. Consider the decomposition of the spectrum of $\mathcal T_{a,\xi}$ in Lemma~\ref{l:31}. The eigenvalues in
$\operatorname{spec}_0(\mathcal T_{a,\xi})$ are the eigenvalues of the restriction of $\mathcal{T}_{a,\xi}$ to the three-dimensional subspace $V_{a,\xi}$. We determine the location of these eigenvalues by computing successively
a basis of  $V_{a,\xi}$, 
the $3\times3$ matrix representing the action of $\mathcal T_{a,\xi}$
on this basis, and the eigenvalues of this matrix. Note that for $a=0$, $V_{0,\xi}$ is spanned by $\{\cos z, \sin z,1/\sqrt{2}\}$. We use the transformation operators $\mathcal{U}_{a,\xi}$ derived in Lemma \ref{l:31} to construct a basis for $V_{a,\xi}$. We give expansions in the following lemma and provide the proof in Appendix~\ref{app:eig}.
\begin{lemma}\label{lem:eigenfunctions} For $a$ and $\xi$ sufficiently small, expansion of a basis for $V_{a,\xi}$ is
    \begin{align*}
     \phi_{1}(z)&=\cos(z)+\frac{1}{k^{2}}\cos(2z)a-\frac{1}{k^{2}}\sin(2z) ia\xi+\frac{1}{16k^4}(9\cos(3z) -20 \cos(z))a^{2} \\
     &\quad +\mathcal{O}((a+\xi)^3),\\
     \phi_{2}(z)&=\sin(z)+\frac{1}{k^{2}}\sin(2z) a+\frac{1}{k^{2}}\cos(2z) i a\xi+\frac{1}{16k^4}(9\sin(3z) -20 \sin(z))a^{2}\\
     &\quad +\mathcal{O}((a+\xi)^3),\\
     \phi_{3}(z)&=\frac{1}{\sqrt{2}}+\mathcal{O}((a+\xi)^3).
\end{align*}
\end{lemma}

To proceed, we expand the operator $\mathcal{T}_{a,\xi}$  by substituting the expansion of $c(a)$ and  $w$ obtained in Theorem~\ref{thm:existence}, we obtain
\begin{equation} \label{ew}
   \mathcal{T}_{a,\xi} = \mathcal{T}_{0,\xi} +  {a}\mathcal{T}_{1} +{a^2} \mathcal{T}_{2} + \mathcal{O}(a^3),
\end{equation}
where $\mathcal{T}_{0,\xi}$ is in \eqref{eq:T0xi} and
\begin{subequations}
\begin{align}
\mathcal{T}_1 &= -15 k^2 (\partial_z + i \xi) \cos(z) \left( (\partial_z + i \xi)^2 - 1 \right), \\
\mathcal{T}_2 &= \frac{15}{2} (\partial_z + i \xi) \left( 11 + 15 (\partial_z + i \xi)^2 - \cos(2z) \left( (\partial_z + i \xi)^2 - 1 \right) \right).
\end{align}
\end{subequations}

In what follows, we present the action of $\mathcal{T}_{0,\xi}$, $\mathcal{T}_1$, and $\mathcal{T}_2$ on $\cos(n z)$ and $\sin(n z)$ for $ n = 1, 2, 3$. We start with $\mathcal{T}_{0,\xi}$ and calculate that
\begin{align*}
 \mathcal{T}_{0,\xi}(\cos(z)) &= 
k^4 \left(-4 i \xi \cos (z) - 10 i \xi^3 \cos (z) - i \xi^5 \cos (z) + 10 \xi^2 \sin (z) + 5 \xi^4 \sin (z) \right),\\
 \mathcal{T}_{0,\xi}(\sin(z)) &=k^4 \left(-10 \xi^2 \cos (z) - 5 \xi^4 \cos( z) - 4 i \xi \sin (z) - 10 i \xi^3 \sin (z) - i \xi^5 \sin (z) \right), \\
 \mathcal{T}_{0,\xi}(\cos(2z)) &=
k^4 \left(-79 i \xi \cos (2z) - 40 i \xi^3 \cos (2z) - i \xi^5 \cos (2z) + 30 \sin (2z) + 80 \xi^2 \sin (2z)\right.\\
&\quad \left. + 10 \xi^4 \sin (2z) \right), \\
 \mathcal{T}_{0,\xi}(\sin(2z)) &= 
k^4 \left(-30 \cos (2z) - 80 \xi^2 \cos (2z)- 10 \xi^4 \cos (2z) - 79 i \xi \sin (2z) - 40 i \xi^3 \sin (2z) \right.\\
&\quad \left.- i \xi^5 \sin (2z) \right), \\
 \mathcal{T}_{0,\xi}(\cos(3z)) &=
k^4 \left(-404 i \xi \cos (3z) - 90 i \xi^3 \cos (3z) - i \xi^5 \cos (3z) + 240 \sin (3z) \right.\\
&\quad \left. + 270 \xi^2 \sin (3z) + 15 \xi^4 \sin (3z) \right), \\
 \mathcal{T}_{0,\xi}(\sin(3z)) &= 
k^4 \left(-240 \cos (3z) - 270 \xi^2 \cos (3z) - 15 \xi^4 \cos (3z) - 404 i \xi \sin (3z) \right.\\
&\quad \left.- 90 i \xi^3 \sin (3z) - i \xi^5 \sin (3z) \right).
 \end{align*}
 
The action of $\mathcal{T}_1$  on $\cos(n z)$ and $\sin(n z)$ for $n = 1, 2$ is
 \begin{align*}
\mathcal{T}_{1}(\cos(z)) &= - 30 k^2 \sin (2z) 
+ \xi (30 i k^2 + 60 i k^2 \cos (2z)) 
+ \xi^2 (-45 k^2 \sin (2z))\\ 
&\quad + \xi^3 (15 i k^2 + 15 i k^2 \cos (2z)), \\
\mathcal{T}_{1}(\sin(z)) &= 30 k^2 \cos (2z) 
+ \xi (60 i k^2 \sin (2z)) 
+ \xi^2 (30 k^2 + 45 k^2 \cos (2z)) 
+ \xi^3 (15 i k^2 \sin (2z)), \\
\mathcal{T}_{1}(\cos(2z)) &= \frac{75}{2} k^2 \sin (z) + \frac{225}{2} k^2 \sin (3z) 
+ \xi \left(105 i k^2 \cos (z) + 165 i k^2 \cos (3z) \right) 
 \\
 &+ \xi^2 \left(-\frac{135}{2} k^2 \sin( z )- \frac{165}{2} k^2 \sin (3z) \right) 
+ \xi^3 \left(15 i k^2 \cos (z) + 15 i k^2 \cos (3z) \right), \\
\mathcal{T}_{1}(\sin(2z)) &= \frac{75}{2} k^2 \cos (z) + \frac{225}{2} k^2 \cos (3z) 
+ \xi \left(105 i k^2 \sin (z) + 165 i k^2 \sin (3z) \right) \\
&\quad + \xi^2 \left(\frac{135}{2} k^2 \cos (z) + \frac{165}{2} k^2 \cos (3z) \right) 
+ \xi^3 \left(15 i k^2 \sin (z )+ 15 i k^2 \sin (3z) \right).
\end{align*}
Finally, we calculate the action of the $\mathcal{T}_2$ on $\cos(z)$ and $\sin(z)$ as
\begin{align*}
\mathcal{T}_{2}(\cos(z)) &= \frac{45}{2} \sin (z) - \frac{45}{2} \sin (3z) 
+ \xi \left(-255 i \cos (z)+ 30 i \cos (3z) \right) \\
&\quad + \xi^2 \left(\frac{1365}{4} \sin (z) - \frac{75}{4} \sin (3z) \right) 
+ \xi^3 \left(-\frac{435}{4} i \cos (z) + \frac{15}{4} i \cos (3z) \right), \\
\mathcal{T}_{2}(\sin(z)) &= -\frac{75}{2} \cos (z) + \frac{45}{2} \cos (3z) 
+ \xi \left(-255 i \sin (z) + 30 i \sin (3z) \right) \\
&\quad + \xi^2 \left(-\frac{1335}{4} \cos (z) + \frac{75}{4} \cos (3z) \right) 
+ \xi^3 \left(-\frac{465}{4} i \sin (z) + \frac{15}{4} i \sin (3z) \right).
\end{align*}

In what follows, we calculate the action of $\mathcal{T}_{a,\xi}$ to eigenfunctions $\phi_{i}(z)$, $i=1,2,3$ in Lemma~\ref{lem:eigenfunctions} up to second order in $a$ and $\xi$ to obtain
\begin{align*}
    \mathcal{T}_{a,\xi}(\phi_{1}(z)) &= -4i k^{4} \cos (z) \xi + 10 k^4 \sin (z) \xi^2  
    + i k^{2} (15 + 4 \cos (2z)) a \xi  \\
    &\quad + (-15 \sin (z) + 315 \sin (3z)) a^2 + O((a+\xi)^3),\\
    \mathcal{T}_{a,\xi}(\phi_{2}(z)) &= -4i k^4 \sin (z) \xi - 10 k^4 \cos(z) \xi^2
    + 4i k^2 \sin(2 z) a\xi\\&\quad - 315 \cos(3z) a^2 + O((a+\xi)^3), \\
\mathcal{T}_{a,\xi}{(\phi_{3})(z)}&=
\frac{1}{\sqrt{2}}\left(i k^4 \xi
- 15 k^2 \sin(z)a +15i k^2 \cos(z)a\xi 
 -15  \sin (2z)a^2\right)\\
 &\quad + O((a+\xi)^3).
\end{align*}
Next, we compute a $3\times 3$ matrix
\[
\mathbb{B}_{a,\xi}=\left(\langle \mathcal{T}_{a,\xi}{(\phi_{i})(z)}, \phi_j(z)\rangle \right)_{i,j=1,2,3}
\]
which captures the action of $\mathcal{T}_{a,\xi}$ on the eigenspace $V_{a,\xi}$, where inner product is given by \eqref{eq:inn}. A straightforward calculation reveals that
\[
\mathbb{B}_{a,\xi} = 
\begin{pmatrix}
-4i k^4 \xi & -15a^2 + 10k^4 \xi^2 & 15i \sqrt{2} k^2 a\xi \\
-10k^4 \xi^2  & -4i k^4 \xi & 0 \\
\frac{15 i \sqrt{2} k^2 a\xi}{2} & \frac{-15\sqrt{2} k^2 a}{2} & i k^4 \xi 
\end{pmatrix}+O((a+\xi)^3)
\]
The characteristic polynomial of $\mathbb{B}_{a,\xi}$ in $\lambda$ is
\begin{align*}
\mathbb{P}(\lambda; a, \xi,k):=&-\lambda^3 - 7i k^4 \xi\, \lambda^2 
- \left(75 a^2 k^4 \xi^2 - 8 k^8 \xi^2 + 100 k^8 \xi^4\right) \lambda \\
&\quad + \left(1200 i a^2 k^8 \xi^3 - 16 i k^{12} \xi^3 + 100 i k^{12} \xi^5\right).
\end{align*}
Setting \( \lambda=i \mu \), $\mathbb{P}(\lambda; a, \xi,k)=i\mathbb{Q}(\mu; a, \xi,k)$ where
\begin{align*}
\mathbb{Q}(\mu; a, \xi,k):=&\mu^3 + 7 k^4 \xi\, \mu^2 
+ \left(-75 a^2 k^4 \xi^2 + 8 k^8 \xi^2 - 100 k^8 \xi^4\right) \mu \\
&+ \left(1200 a^2 k^8 \xi^3 - 16 k^{12} \xi^3 + 100 k^{12} \xi^5\right)
\end{align*}
The discriminant of $\mathbb{Q}(\mu; a, \xi,k)$ is
\begin{align*}
\Delta &= 15625\, k^{12} \xi^6 \Big(
108\, a^6
- 3231\, a^4 k^4
+ 48\, a^2 k^8
+ 432\, a^4 k^4 \xi^2
- 1488\, a^2 k^8 \xi^2\\
&\quad + 16\, k^{12} \xi^2
+ 576\, a^2 k^8 \xi^4
- 128\, k^{12} \xi^4
+ 256\, k^{12} \xi^6
\Big)
\end{align*}
For small $|a|$ and $|\xi|$, the sign of $\Delta$ will be determined by the leading term $k^{4} \xi^2+3 a^2$, which is always positive.
Therefore, for sufficiently small $|a|$ and $|\xi|$, all roots of $\mathbb{Q}(\mu; a, \xi,k)$ are real and hence, all roots of characteristic polynomial, $\mathbb{P}(\lambda; a, \xi,k)$, are purely imaginary. This completes the proof of Theorem~\ref{thm:spec}.



\appendix

\section{Proof of Lemma~\ref{lem:eigenfunctions}}\label{app:eig}
We expand $\mathcal{U}_{a,\xi}$ in \eqref{e:u} in $a$ and $\xi$, and evaluate it on $\cos(z)$, $\sin(z)$, and $\frac{1}{\sqrt{2}}$ to obtain $\phi_i$, $i=1,2,3$ in Lemma~\ref{lem:eigenfunctions}. 
By applying the Neumann series to $(\mathcal{I}-(\mathcal{P}_{a,\xi}-\mathcal{P}_{0,0})^2)^{-1/2}$, \eqref{e:u} can be expressed as
 \begin{equation*}\label{e:ue}
\mathcal{U}_{a,\xi}\mathcal{P}_{0,0}=\mathcal{P}_{a,\xi}\mathcal{P}_{0,0}+\dfrac{1}{2}(\mathcal{P}_{a,\xi}-\mathcal{P}_{0,0})^2\mathcal{P}_{a,\xi}\mathcal{P}_{0,0}+O(\mathcal{P}_{a,\xi}-\mathcal{P}_{0,0})^4
 \end{equation*}
Using this, it follows that
\begin{equation*}
    \mathcal{U}_{0,0} \mathcal{P}_{0,0} = \mathcal{P}_{0,0}, \quad \mathcal{U}'_{0,0} \mathcal{P}_{0,0} = \mathcal{P}'_{0,0} \mathcal{P}_{0,0}, \quad \dot{\mathcal{U}}_{0,0} \mathcal{P}_{0,0} = \dot{\mathcal{P}}_{0,0} \mathcal{P}_{0,0},
\end{equation*}
\begin{equation*}
    \mathcal{U}''_{0,0} \mathcal{P}_{0,0} = \mathcal{P}''_{0,0} \mathcal{P}_{0,0}, \quad \ddot{\mathcal{U}}_{0,0} \mathcal{P}_{0,0} = \ddot{\mathcal{P}} _{0,0} \mathcal{P}_{0,0}, \quad  \dot{\mathcal{U}}'_{0,0} \mathcal{P}_{0,0} = \dot{\mathcal{P}}'_{0,0}\mathcal{P}_{0,0} - \frac{1}{2} \mathcal{P}_{0,0} \dot{\mathcal{P}}'_{0,0} \mathcal{P}_{0,0},
\end{equation*}
where $'$ and $\cdot$ are derivatives with respect to $a$ and $\xi$ respectively. From \eqref{eq:P}, we obtain that
\begin{align*}
    \mathcal{P}'_{0,0} &= \frac{1}{2\pi i} \oint_{\Gamma} (\mathcal{T}_{0,0} - \lambda)^{-1}\mathcal{T}'_{0,0} (\mathcal{T}_{0,0} - \lambda)^{-1} d\lambda,\\
    \dot{\mathcal{P}}_{0,0} &= \frac{1}{2\pi i} \oint_{\Gamma} (\mathcal{T}_{0,0} - \lambda)^{-1} \dot{\mathcal{T}}_{0,0} (\mathcal{T}_{0,0} - \lambda)^{-1} d\lambda,\\
    \dot{\mathcal{P}'}_{0,0} &= -\frac{1}{2\pi i} \oint_{\Gamma} (\mathcal{T}_{0,0} - \lambda)^{-1} \dot{\mathcal{T}}_{0,0} (\mathcal{T}_{0,0} - \lambda)^{-1} \mathcal{T}_{0,0}' (\mathcal{T}_{0,0} - \lambda)^{-1} d\lambda\\
   &\quad - \frac{1}{2\pi i} \oint_{\Gamma} (\mathcal{T}_{0,0} - \lambda)^{-1} \mathcal{T}'_{0,0} (\mathcal{T}_{0,0} - \lambda)^{-1} \dot{\mathcal{T}}_{0,0} (\mathcal{T}_{0,0} - \lambda)^{-1} d\lambda\\
    &\quad + \frac{1}{2\pi i} \oint_{\Gamma} (\mathcal{T}_{0,0} - \lambda)^{-1} \dot{\mathcal{T}}'_{0,0} (\mathcal{T}_{0,0} - \lambda)^{-1} d\lambda,\\
    \mathcal{P}''_{0,0} &= -\frac{1}{\pi i} \oint_{\Gamma} (\mathcal{T}_{0,0} - \lambda)^{-1} \mathcal{T}'_{0,0} (\mathcal{T}_{0,0} - \lambda)^{-1} \mathcal{T}_{0,0}' (\mathcal{T}_{0,0} - \lambda)^{-1} d\lambda\\
    &\quad + \frac{1}{2\pi i} \oint_{\Gamma} (\mathcal{T}_{0,0} - \lambda)^{-1} \mathcal{T}''_{0,0} (\mathcal{T}_{0,0} - \lambda)^{-1} d\lambda.
\end{align*}
and
\begin{align*}
   \ddot{\mathcal{P}}_{0,0} &= -\frac{1}{\pi i} \oint_{\Gamma} (\mathcal{T}_{0,0} - \lambda)^{-1} \dot{\mathcal{T}}_{0,0} (\mathcal{T}_{0,0} - \lambda)^{-1} \dot{\mathcal{T}}_{0,0} (\mathcal{T}_{0,0} - \lambda)^{-1} d\lambda\\
    &\quad + \frac{1}{2\pi i} \oint_{\Gamma} (\mathcal{T}_{0,0} - \lambda)^{-1} \ddot{\mathcal{T}}_{0,0} (\mathcal{T}_{0,0} - \lambda)^{-1} d\lambda.
\end{align*}
where $\Gamma:=\partial B(0;R/3)$ as given in Lemma~\ref{l:31}. Moreover, derivatives of $\mathcal{T}_{a,\xi}$ appearing in the above integrals can be calculated from expression of $\mathcal{T}_{a,\xi}$ in \eqref{eq:Ta}-\eqref{E:Lxi} using expansions of $w$ and $c$ from Theorem~\ref{thm:existence} as
\begin{align*}
    \mathcal{T}_{0,0}'&=-15k^2\partial_z(\cos(z))(\partial^2_{z}-1), \\
    \dot{\mathcal{T}}_{0,0}&=ik^4(1-5\partial^{4}_{z}),\\
    \dot{\mathcal{T}}_{0,0}'&=-i15k^2(3\cos(z)\partial^2_{z}-2\sin(z)\partial_{z}-\cos(z)),\\
    \mathcal{T}_{0,0}''&=\frac{15}{2}\partial_{z}(11+15\partial^2_{z}-\cos(2z)(\partial^2_{z}-1)),\\
   \ddot{ \mathcal{T}}_{0,0}&=10k^4\partial^3_{z}.
\end{align*}
From \eqref{eq:T0ei}, we observe that the operator \( \mathcal{T}_{0,0} \) acts on \( e^{inz} \) as
\[
\mathcal{T}_{0,0} e^{inz} = i \omega_{n,0} e^{inz}.
\]
Using this, we can compute the action of \( \mathcal{T}_{0,0} \) on \( \cos(nz) \) and \( \sin(nz) \) as
\begin{align*}
    \mathcal{T}_{0,0} (\cos(nz)) &= -\omega_{n,0} \sin(nz), \\
    \mathcal{T}_{0,0} (\sin(nz)) &= \omega_{n,0} \cos(nz).
\end{align*}
Next, we examine the action of the inverse operator \( (\mathcal{T}_{0,0} - \lambda)^{-1} \) on \( \cos(nz) \) and \( \sin(nz) \). Let us express the action on \( \cos(nz) \) and \( \sin(nz) \) as follows
\begin{align*}
    (\mathcal{T}_{0,0} - \lambda)^{-1} \cos (nz )&= \mathcal{A} \cos (nz) + \mathcal{B} \sin (nz), \\
    (\mathcal{T}_{0,0} - \lambda)(\mathcal{A} \cos (nz) + \mathcal{B} \sin (nz)) &= \cos (nz).
\end{align*}
This leads to the following system of equations
\begin{align*}
    \mathcal{A} \omega_{n,0} + \mathcal{B} \lambda &= 0, \\
    -\mathcal{A} \lambda + \mathcal{B}\omega_{n,0} &= 1.
\end{align*}
Solving this system yields the values for \( \mathcal{A} \) and \( \mathcal{B} \)
\[
\mathcal{A} = -\frac{\lambda}{\omega_{n,0} ^2+ \lambda^2}, \quad
\mathcal{B} = \frac{\omega_{n,0}}{ \omega_{n,0} ^2+ \lambda^2}.
\]
Thus, the action of the operator \( (\mathcal{T}_{0,0} - \lambda)^{-1} \) can be expressed as
\begin{equation}\label{T00}
\begin{aligned}
    (\mathcal{T}_{0,0} - \lambda)^{-1} \cos(nz) &= -\frac{\lambda}{\omega_{n,0}^2 + \lambda^2} \cos(nz) 
    + \frac{\omega_{n,0}}{\omega_{n,0}^2 + \lambda^2} \sin(nz), \\
    (\mathcal{T}_{0,0} - \lambda)^{-1} \sin(nz) &= \frac{\omega_{-n,0}}{\omega_{n,0}^2 + \lambda^2} \cos(nz) 
    - \frac{\lambda}{\omega_{n,0}^2+ \lambda^2} \sin(nz).
\end{aligned}
\end{equation}

Now, we are ready to compute the coefficients of \(a, \xi, a\xi, \xi^{2},\text{and} \ a^{2}\) in the expansions of $\phi_i$, $i=1,2,3$. For example, let coefficient of a in $\phi_1$ be $\phi_1^a$, then
\begin{align*}
    \phi^{a}_{1}(z) &= \mathcal{P}'_{0,0} (\cos(z)) \\
    &= \frac{1}{2\pi i} \oint_{\Gamma} (\mathcal{T}_{0, 0}  - \lambda)^{-1}
    \mathcal{T}'_{0,0} (\mathcal{T}_{0, 0}  - \lambda)^{-1} (\cos(z)) \, d\lambda.
\end{align*}
Applying the operator action,
\begin{align*}
    (\mathcal{T}_{0, 0}  - \lambda)^{-1}(\cos(z)) &=\frac{-\cos(z)}{\lambda},\\
    \mathcal{T}'_{0,0} (\mathcal{T}_{0, 0}  - \lambda)^{-1}(\cos(z)) &=\frac{30k^{2}}{\lambda}\sin(2z),
\end{align*}
from the equation \ref{T00}
\begin{align*}
    (\mathcal{T}_{0, 0}-\lambda)^{-1}(\sin(2z))&=\frac{-\lambda}{\lambda^{2}+900k^{8}}\sin(2z)+ \frac{30k^{4}}{\lambda^{2}+900k^{8}}\cos(2z)
\end{align*}
Substituting these into our integral and using partial fraction, we get
\begin{align*}
    \phi_{1}^{a}(z) &= \frac{\cos2z}{k^{2}}.
\end{align*}
Similarly, we compute the remaining coefficients in eigenfunctions $\phi_i$, $i=1,2,3$ and obtain their expressions as given in Lemma~\ref{lem:eigenfunctions}.

\bibliographystyle{amsalpha} 
\bibliography{stability.bib}

@article{konno1992conservation,
  title={Conservation laws of modified {Sawada--Kotera} equation in complex plane},
  author={Konno, K.},
  journal={Journal of the Physical Society of Japan},
  volume={61},
  number={1},
  pages={51--54},
  year={1992},
  publisher={Physical Society of Japan}
}

@article{dye2001bidirectional,
  title={On bidirectional fifth-order nonlinear evolution equations, Lax pairs and directionally dependent solitary waves},
  author={Dye, J. M. and Parker, A.},
  journal={Journal of Mathematical Physics},
  volume={42},
  number={6},
  pages={2567--2589},
  year={2001},
  publisher={AIP Publishing}
}

@article{IoossKirch1990,
  author  = {G{\'e}rard Iooss and Klaus Kirchg{\"a}ssner},
  title   = {Water waves for small surface tension: An approach via normal form},
  journal = {Proceedings of the Royal Society of Edinburgh Section A: Mathematics},
  volume  = {122},
  number  = {1-2},
  pages   = {267--299},
  year    = {1990}
}

@article{sawada1974method,
title={A method for finding {N}-soliton solutions of the {K}d{V} equation and {K}d{V}-like equation},
  author={Sawada, K and Kotera, T},
  journal={Progress of Theoretical Physics},
  volume={51},
  pages={1355--1367},
  year={1974}
}

@article{dodd1978prolongation,
  title={The prolongation structure of the {CDG} equation},
  author={Dodd, Robert K and Gibbon, John D},
  journal={Proceedings of the Royal Society of London. A. Mathematical and Physical Sciences},
  volume={361},
  pages={219--237},
  year={1978}
}

@article{tan2002semi,
  title={Semi-stability of embedded solitons in the general fifth-order {KdV} equation},
  author={Tan, Yu and Yang, Jianke and Pelinovsky, Dmitry E},
  journal={Wave Motion},
  volume={36},
  number={3},
  pages={241--255},
  year={2002},
  publisher={Elsevier}
}

@article{kapitula2015spectral,
  title={On the spectral and orbital stability of spatially periodic stationary solutions of generalized Korteweg--de Vries equations},
  author={Kapitula, Todd and Deconinck, Bernard},
  journal={Hamiltonian partial differential equations and applications},
  volume={75},
  pages={285--322},
  year={2015},
  publisher={Springer}
}

@article{wang2023spectral,
  title={Spectral stability of multi-solitons for generalized Hamiltonian system I: The Caudrey--Dodd--Gibbon--Sawada--Kotera equation},
  author={Wang, Zhong},
  journal={Physica D: Nonlinear Phenomena},
  volume={444},
  pages={133610},
  year={2023},
  publisher={Elsevier}
}

@article{haragus2006spectral,
  title={Spectral stability of wave trains in the {K}awahara equation},
  author={Haragus, Mariana and Lombardi, Eric and Scheel, Arnd},
  journal={Journal of Mathematical Fluid Mechanics},
  volume={8},
  pages={482--509},
  year={2006},
  publisher={Springer}
}

@article{creedon2024existence,
  title={Existence of All Wilton Ripples of the {Kawahara} Equation},
  author={Creedon, Ryan P},
  journal={arXiv preprint arXiv:2411.13508},
  year={2024}
}

@article{creedon2021high,
  title={High-frequency instabilities of the {Kawahara} equation: a perturbative approach},
  author={Creedon, Ryan and Deconinck, Bernard and Trichtchenko, Olga},
  journal={SIAM Journal on Applied Dynamical Systems},
  volume={20},
  number={3},
  pages={1571--1595},
  year={2021},
  publisher={SIAM}
}

@article{esfahani2021existence,
  title={Existence and stability of traveling waves of the fifth-order {K}d{V} equation},
  author={Esfahani, Amin and Levandosky, Steven},
  journal={Physica D: Nonlinear Phenomena},
  volume={421},
  pages={132872},
  year={2021},
  publisher={Elsevier}
}

@article{il1992stability,
  title={Stability of solitary waves in dispersive media described by a fifth-order evolution equation},
  author={Il'Ichev, AT and Semenov, A Yu},
  journal={Theoretical and Computational Fluid Dynamics},
  volume={3},
  number={6},
  pages={307--326},
  year={1992},
  publisher={Springer}
}

@article{dey2000stability,
  title={Stability of compacton solutions of fifth-order nonlinear dispersive equations},
  author={Dey, Bishwajyoti and Khare, Avinash},
  journal={Journal of Physics A: Mathematical and General},
  volume={33},
  number={30},
  pages={5335},
  year={2000},
  publisher={IOP Publishing}
}

@article{quintero2016analytic,
  title={Analytic and numerical nonlinear stability/instability of solitons for a {K}awahara-like model},
  author={Quintero, Jos{\'e} R and Mu{\~n}oz, Juan Carlos},
  journal={Analysis and Applications},
  volume={14},
  number={04},
  pages={479--501},
  year={2016},
  publisher={World Scientific}
}

@article{trichtchenko2018stability,
  title={Stability of periodic traveling wave solutions to the {Kawahara} equation},
  author={Trichtchenko, Olga and Deconinck, Bernard and Koll{\'a}r, Richard},
  journal={SIAM Journal on Applied Dynamical Systems},
  volume={17},
  number={4},
  pages={2761--2783},
  year={2018},
  publisher={SIAM}
}

@article{de2017orbital,
  title={Orbital stability of periodic traveling wave solutions for the {Kawahara} equation},
  author={de Andrade, Thiago Pinguello and Crist{\'o}fani, Fabr{\'\i}cio and Natali, F{\'a}bio},
  journal={Journal of Mathematical Physics},
  volume={58},
  number={5},
  year={2017},
  publisher={AIP Publishing}
}

@article{weiss1984classes,
  title={On classes of integrable systems and the Painlev{\'e} property},
  author={Weiss, John},
  journal={Journal of Mathematical Physics},
  volume={25},
  number={1},
  pages={13--24},
  year={1984},
  publisher={American Institute of Physics}
}

@article{bridges2002linear,
  title={Linear instability of solitary wave solutions of the {Kawahara} equation and its generalizations},
  author={Bridges, Thomas J and Derks, Gianne},
  journal={SIAM Journal on Mathematical Analysis},
  volume={33},
  number={6},
  pages={1356--1378},
  year={2002},
  publisher={SIAM}
}

@article{kabakouala2018stability,
  title={On the stability of the solitary waves to the (generalized) {K}awahara equation},
  author={Kabakouala, Andr{\'e} and Molinet, Luc},
  journal={Journal of Mathematical Analysis and Applications},
  volume={457},
  number={1},
  pages={478--497},
  year={2018},
  publisher={Elsevier}
}

@article{bridges2002stability,
  title={Stability and instability of solitary waves of the fifth-order {KdV} equation: a numerical framework},
  author={Bridges, Thomas J and Derks, Gianne and Gottwald, Georg},
  journal={Physica D: Nonlinear Phenomena},
  volume={172},
  number={1-4},
  pages={190--216},
  year={2002},
  publisher={Elsevier}
}

@article{natali2010note,
  title={A note on the stability for {K}awahara--{K}d{V} type equations},
  author={Natali, F{\'a}bio},
  journal={Applied mathematics letters},
  volume={23},
  number={5},
  pages={591--596},
  year={2010},
  publisher={Elsevier}
}

@article{Kaup1980,
  author  = {D. J. Kaup},
  title   = {On the inverse scattering problem for cubic eigenvalue problems of the class \(\psi_{xxx} + 6Q(x)\psi_x + 6R(x)\psi = \lambda\psi\)},
  journal = {Studies in Applied Mathematics},
  volume  = {62},
  number  = {3},
  pages   = {189--216},
  year    = {1980},
  doi     = {10.1002/sapm1980623189}
}

@article {G,
    AUTHOR = {Gardner, Robert A.},
     TITLE = {Spectral analysis of long wavelength periodic waves and
              applications},
   JOURNAL = {J. Reine Angew. Math.},
  FJOURNAL = {Journal f\"ur die Reine und Angewandte Mathematik. [Crelle's
              Journal]},
    VOLUME = {491},
      YEAR = {1997},
     PAGES = {149--181},
      ISSN = {0075-4102},
     CODEN = {JRMAA8},
   MRCLASS = {35Qxx (35B10 35B35)},
  MRNUMBER = {1476091 (99a:35196)},
MRREVIEWER = {Joel Smoller},
       DOI = {10.1515/crll.1997.491.149},
       URL = {http://dx.doi.org/10.1515/crll.1997.491.149},
}

@article{DK,
AUTHOR = {Deconinck, Bernard and Kapitula, Todd},
TITLE = {The orbital stability of the cnoidal waves of the {K}orteweg-de {V}ries equation},
JOURNAL = {Phys. Lett. A},
FJOURNAL = {Physics Letters. A},
VOLUME = {374},
YEAR = {2010},
NUMBER = {39},
PAGES = {4018--4022},
ISSN = {0375-9601},
CODEN = {PYLAAG},
MRCLASS = {35Q53 (35B35)},
MRNUMBER = {2683991 (2011f:35297)},
DOI = {10.1016/j.physleta.2010.08.007},
URL = {http://dx.doi.org/10.1016/j.physleta.2010.08.007},
}

@article{P,
AUTHOR = {Angulo Pava, Jaime},
TITLE = {Nonlinear stability of periodic traveling wave solutions to the {S}chr\"odinger and the modified {K}orteweg-de {V}ries equations},
JOURNAL = {J. Differential Equations},
FJOURNAL = {Journal of Differential Equations},
VOLUME = {235},
YEAR = {2007},
NUMBER = {1},
PAGES = {1--30},
ISSN = {0022-0396},
CODEN = {JDEQAK},
MRCLASS = {35Q53 (34C25 35B10 35B35 37K45)},
MRNUMBER = {2309564 (2008d:35189)},
MRREVIEWER = {Bernard Deconinck},
DOI = {10.1016/j.jde.2007.01.003},
URL = {http://dx.doi.org/10.1016/j.jde.2007.01.003},
}

@article {S,
    AUTHOR = {Serre, Denis},
     TITLE = {Spectral stability of periodic solutions of viscous
              conservation laws: large wavelength analysis},
   JOURNAL = {Comm. Partial Differential Equations},
  FJOURNAL = {Communications in Partial Differential Equations},
    VOLUME = {30},
      YEAR = {2005},
    NUMBER = {1-3},
     PAGES = {259--282},
      ISSN = {0360-5302},
     CODEN = {CPDIDZ},
   MRCLASS = {35L65 (35B10 35B35)},
  MRNUMBER = {2131054 (2006f:35178)},
MRREVIEWER = {Kevin R. Zumbrun},
       DOI = {10.1081/PDE-200044492},
       URL = {http://dx.doi.org/10.1081/PDE-200044492},
}

@article{AS,
author = {Ablowitz,Mark J. and Segur,Harvey},
title = {On the evolution of packets of water waves},
journal = {Journal of Fluid Mechanics},
volume = {92},
issue = {04},
month = {6},
month = {6},
year = {1979},
issn = {1469-7645},
pages = {691--715},
numpages = {25},
doi = {10.1017/S0022112079000835},
URL = {http://journals.cambridge.org/article_S0022112079000835},
}

@article{caudrey1976new,
  title={A new hierarchy of {K}orteweg--de {V}ries equations},
  author={Caudrey, P. J. and Dodd, R. K. and Gibbon, J. D.},
  journal={Proceedings of the Royal Society of London. A. Mathematical and Physical Sciences},
  volume={351},
  number={1666},
  pages={407--422},
  year={1976}
}

@article{satsuma1977backlund,
  title={A {Bäcklund} transformation for a higher order {Korteweg-de} {V}ries equation},
  author={Satsuma, J. and Kaup, D. J.},
  journal={Journal of the Physical Society of Japan},
  volume={43},
  number={2},
  pages={692--697},
  year={1977}
}

@article{aiyer1986solitons,
  title={Solitons and discrete eigenfunctions of the recursion operator of non-linear evolution equations: I. The {{C}audrey--Dodd--Gibbon--Sawada--Kotera} equations},
  author={Aiyer, R. N. and Fuchssteiner, B. and Oevel, W.},
  journal={Journal of Physics A: Mathematical and General},
  volume={19},
  number={18},
  pages={3755--3770},
  year={1986}
}

@article{fuchssteiner1982bi,
  title={The bi-Hamiltonian structure of some nonlinear fifth-and seventh-order differential equations and recursion formulas for their symmetries and conserved covariants},
  author={Fuchssteiner, Benno and Oevel, Walter},
  journal={Journal of Mathematical Physics},
  volume={23},
  number={3},
  pages={358--363},
  year={1982},
  publisher={American Institute of Physics}
}

@article{ma2023phase,
  title={Phase transition from soliton to breather, soliton-breather molecules, breather molecules of the {Caudrey-Dodd-Gibbon} equation},
  author={Ma, Yu-Lan and Wazwaz, Abdul-Majid and Li, Bang-Qing},
  journal={Physics Letters A},
  volume={488},
  pages={129132},
  year={2023},
  publisher={Elsevier}
}

@article{ramirez2016reductions,
  title={Reductions of PDEs to second order ODEs and symbolic computation},
  author={Ram{\'\i}rez, Jos{\'e} and Romero, JL and Muriel, C},
  journal={Applied Mathematics and Computation},
  volume={291},
  pages={122--136},
  year={2016},
  publisher={Elsevier}
}

@article{li2004different,
  title={Different-periodic travelling wave solutions for nonlinear equations},
  author={Li-Jun, Ye and Ji, Lin},
  journal={Communications in Theoretical Physics},
  volume={41},
  number={4},
  pages={481},
  year={2004},
  publisher={IOP Publishing}
}

@article{tian2010riemann,
  title={Riemann theta functions periodic wave solutions and rational characteristics for the nonlinear equations},
  author={Tian, Shou-fu and Zhang, Hong-qing},
  journal={Journal of Mathematical Analysis and Applications},
  volume={371},
  number={2},
  pages={585--608},
  year={2010},
  publisher={Elsevier}
}

@article{alam2021determination,
  title={Determination of the rich structural wave dynamic solutions to the {Caudrey--Dodd--Gibbon} equation and the {Lax} equation},
  author={Alam, Md Khorshed and Hossain, Md Dulal and Akbar, M Ali and Gepreel, Khaled A},
  journal={Letters in Mathematical Physics},
  volume={111},
  pages={1--19},
  year={2021},
  publisher={Springer}
}

@article{lou1993twelve,
  title={Twelve sets of symmetries of the Caudrey-Dodd-Gibbon-Sawada-Kotera equation},
  author={Lou, Sen-yue},
  journal={Physics Letters A},
  volume={175},
  number={1},
  pages={23--26},
  year={1993},
  publisher={Elsevier}
}

@article{wazwaz2008n,
  title={N-soliton solutions for the combined {KdV--CDG} equation and the {KdV--Lax} equation},
  author={Wazwaz, Abdul-Majid},
  journal={Applied mathematics and computation},
  volume={203},
  number={1},
  pages={402--407},
  year={2008},
  publisher={Elsevier}
}

@article{masperogkdv,
  author  = {Alberto Maspero and Antonio Milosh Radakovic},
  title   = {Full Description of Benjamin--Feir Instability for Generalized Korteweg--de Vries Equations},
  year    = {2024},
journal = {""},
eprint  = {arXiv:2404.06172},
archivePrefix = {arXiv},
primaryClass  = {math.AP}
}

\end{document}